\documentclass[review]{elsarticle}
\usepackage{amsmath}
\usepackage{graphicx,setspace}
\usepackage{amsfonts,amsmath, amssymb, amsthm}
\usepackage{mathrsfs}
\usepackage{epstopdf}
\usepackage{float}
\usepackage{lineno,hyperref}
\usepackage{color}
\usepackage{subfigure}
\usepackage{bm}
\usepackage{graphicx}
\usepackage{amssymb, amsthm}
\usepackage{mathrsfs}
\usepackage{epstopdf}
\usepackage{float}
\usepackage{lineno,hyperref}
\usepackage{color}
\usepackage{subfigure}
\usepackage{cases}

\modulolinenumbers[5]
\numberwithin{equation}{section}
\journal{Physica D: Nonlinear Phenomena}

\begin{document}

\newtheorem{definition}{Definition}
\newtheorem{lemma}{Lemma}
\newtheorem{remark}{Remark}
\newtheorem{theorem}{Theorem}
\newtheorem{proposition}{Proposition}
\newtheorem{assumption}{Assumption}
\newtheorem{example}{Example}
\newtheorem{corollary}{Corollary}
\def\e{\varepsilon}
\def\Rn{\mathbb{R}^{n}}
\def\Rm{\mathbb{R}^{m}}
\def\E{\mathbb{E}}
\def\hte{\bar\theta}
\def\cC{{\mathcal C}}
\numberwithin{equation}{section}

\begin{frontmatter}

\title{Effective Wave Factorization for a Stochastic Schr\"{o}dinger Equation}

\author{\bf\normalsize{
Ao Zhang\footnote{School of Mathematics and Statistics, \& Center for Mathematical Sciences,  Huazhong University of Sciences and Technology, Wuhan 430074,  China. Email: \texttt{zhangao1993@hust.edu.cn}},
Jinqiao Duan\footnote{Department of Applied Mathematics, Illinois Institute of Technology, Chicago, IL 60616, USA. Email: \texttt{duan@iit.edu}}
}}

\begin{abstract}
We study the homogenization of a stochastic Schr\"odinger equation with a large periodic potential in solid state physics. Denoting by $\varepsilon$ the period, the potential is scaled as $\varepsilon^{-2}$. Under a generic assumption on the spectral properties of the associated cell problem, we prove that the solution can be approximately factorized as the product of a fast oscillating cell eigenfunction and of a slowly varying solution of an effective equation. Our method is based on two-scale convergence and Bloch waves theory.
\end{abstract}


\end{frontmatter}

\section{Introduction}

There is a vast literature on periodic and quasi-periodic homogenization of partial differential equations. For nonlinear Schr\"odinger-type equations with semi-classical scaling and an additional highly oscillatory periodic potential, recently the rigorous study of the corresponding asymptotic regime $\varepsilon\to 0$, known as the semi-classical approximation, attracted lots of interest(see e.g. \cite{Di02, Po96, Bu87, Gu88}). The scaling of deterministic Schr\"odinger equation, which is different from the semi-classical scaling, is studied by Allaire and Piatnitski in \cite{Al04, Al05}. Contrasted with deterministic homogenization, very few results are available as regards the homogenization of stochastic partial differential equations(SPDEs) (see \cite{Wa07, WD07}). In some circumstances, randomness has to be taken into account and it often occurs through a random potential. So we consider the macroscopic potential under random perturbation, and study the homogenized problem of stochastic Schr\"odinger equation. To stochastic Schr\"odinger equation, Bouard and Debussche studied the properties of the solutions in the case of additive noise in \cite{de03}, multiplicative noise in \cite{de99} and white noise dispersion in \cite{de10}. Variational solutions of stochastic Schr\"odinger equations were studied in \cite{Ke15, Ke16, KD15}.

To prove the convergence to the homogenized problem of the stochastic Schr\"odinger equation, the main idea is to use Bloch wave theory to build adequate oscillating test functions and to pass to the limit using two-scale convergence. The method of Bloch waves \cite{Blo28}, or the Bloch transform, is a generalization of Fourier transform that leaves invariant periodic functions, for a modern treatment of this topic see \cite{Al04, Al05}. The method of two-scale convergence is a powerful tool for studying homogenization problems for partial differential equations with periodically oscillating coefficients.Two-scale convergence has been introduced by Nguetseng \cite{Ngu98} and Allaire \cite{Al92}. And the theory of the two scale convergence from the periodic to the stochastic setting has been extended by Bourgeat, Mikeli\'c and Wright in \cite{BMW94}, using techniques from ergodic theory. A striking advantage of the two-scale convergence method is that the homogenized and local problems appear directly as convergence results and do not have to be derived by tedious and somewhat dubious calculations. In practice, multiplying the global equation by an adequate test function and applying theorems yields both the local and the homogenized equations, and the proof of the convergence.

We study the homogenization of the following Schr\"odinger equation with white noise
\begin{equation}\label{Equ.1}
\begin{cases}
$$i\frac{\partial u_{\varepsilon}}{\partial t} - \frac{\partial}{\partial x}(\sigma(\frac{x}{\varepsilon})\frac{\partial u_{\varepsilon}}{\partial x} )+(\varepsilon^{-2}c(\frac{x}{\varepsilon})+d(x,\frac{x}{\varepsilon}))u_{\varepsilon}+g(t, \frac{x}{\varepsilon}, u_{\varepsilon})\frac{dW(t)}{dt}=0  \quad \text{in} \ D\times [0,T], \\
u_\varepsilon=0 \quad \text{on} \ \partial D\times [0,T], \\
u_\varepsilon(0,x)=u^0_{\varepsilon}(x)  \quad \text{in} \ D,$$
\end{cases}
\end{equation}
where $D\subset\mathbb{R}$ is an open set, $0<T<\infty$, and the unknown function $u_{\varepsilon}$ is complex-valued. The coefficients $\sigma(y)$, $c(y)$ and $d(x, y)$ are real and bounded functions defined for $x\in D$ and $y\in {\mathbb{T}}$ (the unit torus). The function $g$ is given different assumptions to get different results. Furthermore, the real-valued Wiener processes $W(t)$ is defined on the complete probability space $(\Omega, \mathcal{F}, \mathbb{P})$ endowed with the canonical filtration $(\mathcal{F}_t)_{t\in[0, T]}$.

We are interested in the behavior of the solution $u_{\varepsilon}(t, x, \omega)$ as $\varepsilon\to 0$. And we first introduce the Bloch or shifted cell problem,
\begin{equation}\nonumber
-(\frac{\partial}{\partial y}+2i\pi\theta)\left(\sigma(y)(\frac{\partial}{\partial y}+2i\pi\theta)\psi_n\right)+c(y)\psi_n=\lambda_n(\theta)\psi_n \quad \text{in}\ \mathbb{T},
\end{equation}
where $\theta\in\mathbb{T}$ is a parameter and $(\lambda_n(\theta), \psi_n(y, \theta))$ is the $n^{th}$ eigenpair. In physical terms, the range of $\lambda_n(\theta)$, as $\theta$ run in $\mathbb{T}$, is a Bloch or conduction band (also called Fermi surface). 
Under some assumptions, we focus on higher energy initial data (or excited states) and consider well-prepared initial data of the type
\begin{equation}
u_{\varepsilon}^0(x)=\psi_n(\frac{x}{\varepsilon},\theta^n)e^{2i\pi\frac{\theta^n\cdot x}{\varepsilon}}v^0(x),
\end{equation}
we shall prove in Theorem 1 and Theorem 2 that the solution of (\ref{Equ.1}) with different type of noise satisfies
\begin{equation}
u_{\varepsilon}(t, x, \omega)\approx e^{i\frac{\lambda_n(\theta^n)t}{\varepsilon^2}}e^{2i\pi\frac{\theta^n\cdot x}{\varepsilon}}\psi_n(\frac{x}{\varepsilon},\theta^n)v(t, x, \omega),
\end{equation}
where $v(t, x, \omega)$ is the unique solution of the corresponding homogenized stochastic Schr\"odinger equation.

The paper is organized as follows. In Section 2 we define the functional spaces, make some assumptions, and introduce some results on Bloch theory and two-scale convergence. In Section 3 we derive all the two-scale limits and pass to the limit in the variational formulation using particular test functions. We obtain the homogenized stochastic Schr\"odinger equation with additive noise. Section 4 is devoted to the derivation of the homogenized stochastic Schr\"odinger equation with multiplicative noise.

\section{Preliminaries}
The inner product in $L^2(D)$ is given by
\begin{equation}\nonumber
(u,v):=\int_D u(x)\bar{v}(x)dx, \quad \text{for all}\ u, v\in L^2(D),
\end{equation}
where $\bar{v}$ is the complex conjugate of $v$, while the inner product in $H^1(D)$ is constituted by
\begin{equation}\nonumber
(u,v)_{H^1}:=\int_D [u(x)\bar{v}(x)+\frac{d}{dx}u(x)\frac{d}{dx}\bar{v}(x)]dx, \quad \text{for all}\ u, v\in H^1(D).
\end{equation}

We make the following assumptions on this stochastic Schr\"odinger equation.
\par
{\bf Hypothesis  H.1. }The coefficients $\sigma(y)$ and $c(y)$ are real measurable bounded periodic functions, i.e. their entries belong to $L^{\infty}(\mathbb{T})$, while $d(x, y)$ is real measurable and bounded with respect to $x$, and periodic continuous with respect to $y$, i.e. its entries belong to $L^{\infty}(D; C(\mathbb{T}))$.

\par
{\bf Hypothesis  H.2. }The function $ \sigma $ is uniformly positive definite, i.e. there exists $\nu > 0$ such that: $\sigma(y)\geq\nu$,  for a.e. $y\in \mathbb{T}$ .
\subsection{Bloch Spectrum}
We recall the so-called Bloch (or shifted) spectral cell equation
\begin{equation}\label{Cell}
-(\frac{\partial}{\partial y}+2i\pi\theta)\left(\sigma(y)(\frac{\partial}{\partial y}+2i\pi\theta)\psi_n\right)+c(y)\psi_n=\lambda_n(\theta)\psi_n \quad \text{in}\ \mathbb{T}.
\end{equation}
which, as a compact self-adjoint complex-valued operator on $L^2(\mathbb{T})$, admits a countable sequence of real increasing eigenvalues $(\lambda_n)_{n\geqslant 1}$ (repeated with their multiplicity) and normalized eigenfunctions $(\psi_n)_{n\geqslant 1}$ with $\|\psi_n\|_{L^2(\mathbb{T})}=1$. The dual parameter $\theta$ is called the Bloch frequency and it runs in the dual cell of $\mathbb{T}$, i.e. by periodicity it is enough to consider $\theta\in\mathbb{T}$.

In the sequel, we shall consider an energy level $n\geq 1$ and a Bloch parameter $\theta^n\in\mathbb{T}$ such that the eigenvalue $\lambda_n(\theta^n)$ satisfies some assumptions. Depending on these precise assumptions we obtain different homogenized limits for the Schr\"odinger equation (\ref{Equ.1}).
\par
{\bf Hypothesis  H.3. }
$\lambda_n(\theta^n)$  is a simple eigenvalue; $\theta^n$ is a critical point of $\lambda_n(\theta)$, i.e. $\frac{\partial\lambda_n}{\partial\theta}(\theta^n)=0$.
\begin{remark}
This assumption of simplicity has two important consequences. First, if $\lambda_n(\theta^n)$ is simple, then it is infinitely differentiable in a vicinity of $\theta^n$. Second, if $\lambda_n(\theta^n)$ is simple, then the limit problem is going to be a single Schr\"odinger equation.
\end{remark}

Under Hypothesis  H.3., it is a classical matter to prove that the $n^{th}$ eigenpair of (\ref{Cell}) is smooth in a neighborhood of $\theta^n$\cite{Ka66}. Introducing the operator $\mathbb{A}_n(\theta)$ defined on $L^2(\mathbb{T})$ by
\begin{equation}\label{Oper.1}
\mathbb{A}_n(\theta)\psi=-(\frac{\partial}{\partial y}+2i\pi\theta)\left( \sigma(y)(\frac{\partial}{\partial y}+2i\pi\theta)\psi\right)+c(y)\psi-\lambda_n(\theta)\psi,
\end{equation}
it is easy to differentiate (\ref{Cell}). The first derivative satisfies
\begin{equation}\label{Oper.2}
\mathbb{A}_n(\theta)\frac{\partial\psi_n}{\partial\theta}=2i\pi \sigma(y)(\frac{\partial}{\partial y}+2i\pi\theta)\psi_n+(\frac{\partial}{\partial y}+2i\pi\theta)(\sigma(y)2i\pi \psi_n)+\frac{\partial\lambda_n}{\partial\theta}(\theta)\psi_n,
\end{equation}
and the second derivative is
\begin{equation}\label{Oper.3}
\begin{split}
\mathbb{A}_n(\theta)\frac{\partial^2\psi_n}{\partial\theta^2}=&4i\pi \sigma(y)(\frac{\partial}{\partial y}+2i\pi\theta)\frac{\partial\psi_n}{\partial\theta}+2(\frac{\partial}{\partial y}+2i\pi\theta)\left(\sigma(y)2i\pi \frac{\partial\psi_n}{\partial\theta}\right)\\
&+2\frac{\partial\lambda_n}{\partial\theta}(\theta)\frac{\partial\psi_n}{\partial\theta}-8\pi^2\sigma(y)\psi_n+\frac{\partial^2\lambda_n}{\partial\theta^2}(\theta)\psi_n.
\end{split}
\end{equation}
Under Hypothesis  H.3., we have $\frac{\partial\lambda_n}{\partial\theta}(\theta^n)=0$, thus equation (\ref{Oper.2}) and (\ref{Oper.3}) simplify for $\theta=\theta^n$ and we find
\begin{equation}
\frac{\partial\psi_n}{\partial\theta}=2i\pi\zeta, \quad  \frac{\partial^2\psi_n}{\partial\theta^2}=-4\pi^2\chi,
\end{equation}
where $\zeta$ is the solution of
\begin{equation}\label{SolutionZeta}
\mathbb{A}_n(\theta^n)\zeta=\sigma(y)(\frac{\partial}{\partial y}+2i\pi\theta^n)\psi_n+(\frac{\partial}{\partial y}+2i\pi\theta^n)(\sigma(y)\psi_n) \quad\text{in}\ \mathbb{T},
\end{equation}
and $\chi$ is the solution of
\begin{equation}\label{SolutionChi}
\mathbb{A}_n(\theta^n)\chi=2\sigma(y)(\frac{\partial}{\partial y}+2i\pi\theta^n)\zeta+2(\frac{\partial}{\partial y}+2i\pi\theta^n)(\sigma(y)\zeta)+2\sigma(y)\psi_n-\frac{1}{4\pi^2}\frac{\partial^2\lambda_n}{\partial\theta^2}(\theta^n)\psi_n \quad\text{in}\ \mathbb{T}.
\end{equation}
There exists a unique solution of (\ref{SolutionZeta}), up to the addition of a multiple of $\psi_n$. Indeed, the right hand side of (\ref{SolutionZeta}) satisfies the required compatibility condition or Fredholm alternative (i.e. it is orthogonal to $\psi_n$) because $\zeta$ is just a multiple of the partial derivative of $\psi_n$ with respect to $\theta$ which necessarily exists. By the same token, there exists a unique solution of (\ref{SolutionChi}), up to the addition of a multiple of $\psi_n$. The compatibility condition of (\ref{SolutionChi}) yields a formula for the value $\frac{\partial^2\lambda_n}{\partial\theta^2}(\theta^n)$, see \cite{Al04, Al05}.

\subsection{Two-Scale Convergence}
We will summarize in this section several results about the two-scale convergence that we will use throughout the paper. For the results stated without proofs, see \cite{Al92, BMW94, BM16}. We denote by $C_{\#}(\mathbb{T})$ the space of functions from $C(\bar{\mathbb{T}})$ that have $\mathbb{T}$-periodic boundary values.
\begin{definition}
We say that a sequence $u_{\varepsilon}\in L^2(\Omega\times[0,T]\times D)$ two scale converges to $u\in L^2(\Omega\times[0,T]\times D\times\mathbb{T})$, and denote this convergence by
\begin{equation}\nonumber
u_{\varepsilon}\xrightarrow{2-s}u \quad \text{in} \  \Omega\times[0,T]\times D,
\end{equation}
if for every $\Psi\in L^2(\Omega\times[0,T]\times D;C_{\#}(\mathbb{T}))$ we have
\begin{equation}\nonumber
\begin{split}
\lim_{\varepsilon\to 0}\int_{\Omega}\int^T_0\int_Du_{\varepsilon}(\omega, t, x)&\Psi(\omega, t, x, \frac{x}{\varepsilon})dxdtd\mathbb{P}\\
&=\int_{\Omega}\int^T_0\int_D\int_{\mathbb{T}}u(\omega, t, x, y)\Psi(\omega, t, x, y)dydxdtd\mathbb{P}.\\
\end{split}
\end{equation}
\end{definition}

The following propositions are of great importance in obtaining the homogenization result.
\begin{proposition}
Assume that the sequence $u_{\varepsilon}$ is uniformly bounded in $L^2(\Omega\times[0,T]\times D)$. Then exists a subsequence, still denoted by $u_{\varepsilon}$, and a limit $u_0(\omega, t, x, y)\in L^2(\Omega\times[0,T]\times D\times\mathbb{T})$ such that
\begin{equation}\nonumber
u_{\varepsilon}(\omega, t, x)\xrightarrow{2-s}u_0(\omega, t, x, y) \quad \text{in} \  \Omega\times[0,T]\times D.
\end{equation}
\end{proposition}

\begin{proposition}
Assume that the sequence $u_{\varepsilon}$ is uniformly bounded in $L^2(\Omega\times[0,T]\times D)$, and the sequence $\varepsilon\nabla u_{\varepsilon}$ is also uniformly bounded in $L^2(\Omega\times[0,T]\times D)$. Then there exists a subsequence, still denoted by $u_{\varepsilon}$, and a limit $u_0(\omega, t, x, y)\in L^2(\Omega\times[0,T]\times D;H^1(\mathbb{T}))$ such that
\begin{equation}\nonumber
\begin{split}
&u_{\varepsilon}(\omega, t, x)\xrightarrow{2-s}u_0(\omega, t, x, y) \quad \text{in} \  \Omega\times[0,T]\times D, \\
&\varepsilon\frac{\partial u_{\varepsilon}(\omega, t, x)}{\partial x}\xrightarrow{2-s}\frac{\partial u_0(\omega, t, x, y)}{\partial y} \quad \text{in} \  \Omega\times[0,T]\times D.\\
\end{split}
\end{equation}
\end{proposition}

Notation.  for any function $\phi(x, y)$ defined on $D\times \mathbb{T}$, we denote by $\phi^{\varepsilon}$ the function $\phi(x, \frac{x}{\varepsilon})$.

\section{Homogenization with Additive Noise}
We now study the homogenization of the following Schr\"odinger equation with additive white noise
\begin{equation}\label{Equ.2}
\begin{cases}
$$i\frac{\partial u_{\varepsilon}}{\partial t} - \frac{\partial}{\partial x}(\sigma(\frac{x}{\varepsilon})\frac{\partial u_{\varepsilon}}{\partial x} )+(\varepsilon^{-2}c(\frac{x}{\varepsilon})+d(x,\frac{x}{\varepsilon}))u_{\varepsilon}+g(t, \frac{x}{\varepsilon})\frac{dW(t)}{dt}=0  \quad \text{in} \ D\times [0,T], \\
u_\varepsilon=0 \quad \text{on} \ \partial D\times [0,T], \\
u_\varepsilon(0,x)=u^0_{\varepsilon}(x)  \quad \text{in} \ D.$$
\end{cases}
\end{equation}
\par
{\bf Hypothesis  H.4. } We assume that the function $g(t, \frac{x}{\varepsilon})$ has the following type,
\begin{equation}\nonumber
g(t, \frac{x}{\varepsilon})=e^{i\frac{\lambda_n(\theta^n)t}{\varepsilon^2}}e^{2i\pi\frac{\theta^n\cdot x}{\varepsilon}}\tilde{g}(\frac{x}{\varepsilon}),
\end{equation}
where $\tilde{g}$ is real measurable bounded periodic function.

We obtain the priori estimates, existence and uniqueness of the variational solution of Schr\"odinger equation (\ref{Equ.2}). For the following results, see \cite{Ke16}.
\begin{lemma}
Assume (H.1., H.2., H.4.). For every $\varepsilon>0$, $u^0_{\varepsilon}\in H^1(D)$, and $T>0$, there exists a unique variational solution $u_{\varepsilon}\in L^2(\Omega;C([0,T]);L^2(D))\bigcap L^2(\Omega\times[0,T];H^1(D))$ of stochastic Schr\"odinger equation (\ref{Equ.2}) in the following sense:
\begin{equation}
\begin{split}
(u_{\varepsilon},v)&=(u^0_{\varepsilon},v)-i\int_0^t(\sigma(\frac{x}{\varepsilon})\frac{\partial u_{\varepsilon}}{\partial x}, \frac{\partial v}{\partial x} )ds+i\int_0^t\left((\varepsilon^{-2}c(\frac{x}{\varepsilon})+d(x,\frac{x}{\varepsilon}))u_{\varepsilon},v\right)ds\\
&+i\int_0^t(g(t, \frac{x}{\varepsilon}),v)dW(s),
\end{split}
\end{equation}
for a.e. $\omega\in\Omega$, all $t\in[0,T]$ and for all $v\in H^1(D)$. Moreover, there exists a constant $C_T$ that depends on $T$ such that
\begin{equation}
\mathbb{E}(\sup_{0\leqslant t\leqslant T}\|u_{\varepsilon}\|^2+\varepsilon^2\int_0^T\|u_{\varepsilon}\|^2_{H^1}dt)<C_T.
\end{equation}
\end{lemma}

\begin{theorem}
Assume (H.1.-H.4.) and that the initial data $u_{\varepsilon}^0\in H^1(D)$ is of the form
\begin{equation}
u_{\varepsilon}^0(x)=\psi_n(\frac{x}{\varepsilon},\theta^n)e^{2i\pi\frac{\theta^n\cdot x}{\varepsilon}}v^0(x),
\end{equation}
with $v^0\in H^1(D)$. The solution of (\ref{Equ.2}) can be written as
\begin{equation}
u_{\varepsilon}(t,x)=e^{i\frac{\lambda_n(\theta^n)t}{\varepsilon^2}}e^{2i\pi\frac{\theta^n\cdot x}{\varepsilon}}v_{\varepsilon}(t, x),
\end{equation}
where $v_{\varepsilon}$ two-scale converges  to $\psi_n(y,\theta^n)v(t, x)$, uniformly on compact time intervals in $\mathbb{R}^+$, and $v$ is the unique
solution of the homogenized Schr\"odinger equation
\begin{equation}\label{Homo.1}
\begin{cases}
$$i\frac{\partial v}{\partial t} - \frac{\partial}{\partial x}(\sigma_n^*\frac{\partial v}{\partial x})+d_n^*(x)v+g^*\frac{dW(t)}{dt}=0  \quad \text{in} \ D\times [0,T], \\
v=0 \quad \text{on} \ \partial D\times [0,T], \\
v(0,x)=v^0(x)  \quad \text{in} \ D.$$
\end{cases}
\end{equation}
with $\sigma_n^*=\frac{1}{8\pi^2}\frac{\partial^2\lambda_n(\theta^n)}{\partial\theta^2}$, $d^*_n(x)=\int_{\mathbb{T}}d(x, y)|\psi_n(y)|^2dy$, and $g_n^*=\int_{\mathbb{T}}\tilde{g}(y)|\psi_n(y)|^2dy$.
\end{theorem}

\begin{proof}
Define a sequence $v_{\varepsilon}$ by
\begin{equation}\nonumber
v_{\varepsilon}(t,x)=u_{\varepsilon}(t, x)e^{-i\frac{\lambda_n(\theta^n)t}{\varepsilon^2}}e^{-2i\pi\frac{\theta^n\cdot x}{\varepsilon}}.
\end{equation}
Since $|v_{\varepsilon}|=|u_{\varepsilon}|$, by Lemma 1, we have
\begin{equation}
\mathbb{E}(\sup_{0\leqslant t\leqslant T}\|v_{\varepsilon}\|^2+\varepsilon^2\int_0^T\|v_{\varepsilon}\|^2_{H^1}dt)<C_T,
\end{equation}
and applying Proposition 2, up to a subsequence, there exists a limit $v^*(\omega, t, x, y)\in L^2(\Omega\times[0, T]\times D; H^1(\mathbb{T}))$ such that
\begin{equation}\label{Pro.2}
\begin{split}
&v_{\varepsilon}(\omega, t, x)\xrightarrow{2-s}v^*(\omega, t, x, y) \quad \text{in} \  \Omega\times[0,T]\times D, \\
&\varepsilon\frac{\partial v_{\varepsilon}(\omega, t, x)}{\partial x}\xrightarrow{2-s}\frac{\partial v^*(\omega, t, x, y)}{\partial y} \quad \text{in} \  \Omega\times[0,T]\times D.\\
\end{split}
\end{equation}
Similarly, by definition of the initial data and Proposition 1,
\begin{equation}\label{Pro.1}
v_{\varepsilon}(0, x)\xrightarrow{2-s}\psi_n(y, \theta^n)v^0(x) \quad \text{in} \  \Omega\times[0,T]\times D.
\end{equation}
First step.\ We multiply (\ref{Equ.2}) by the complex conjugate of
 \begin{equation}
\varepsilon^2\phi(\omega, t, x, \frac{x}{\varepsilon}) e^{i\frac{\lambda_n(\theta^n)t}{\varepsilon^2}}e^{2i\pi\frac{\theta^n\cdot x}{\varepsilon}},
\end{equation}
where $\phi(\omega, t, x, \frac{x}{\varepsilon})$ is a smooth test function defined on $\Omega\times[0, T]\times D\times\mathbb{T}$, with compact support in $[0, T]\times D$. 
Integrating with respect to $\omega\in\Omega$ and $t\in[0,T]$, we obtain
\begin{equation}
\begin{split}
i\varepsilon^2&\int_{\Omega}\int_Du^0_{\varepsilon}\bar{\phi}^{\varepsilon}e^{-2i\pi\frac{\theta^n\cdot x}{\varepsilon}}dxd\mathbb{P}-i\varepsilon^2\int_{\Omega}\int_0^T\int_Dv_{\varepsilon}\frac{\partial\bar{\phi}^{\varepsilon}}{\partial t}dxdtd\mathbb{P}\\
&+\int_{\Omega}\int_0^T\int_D\sigma^{\varepsilon}(\varepsilon\frac{\partial}{\partial x}+2i\pi\theta^n)v_{\varepsilon}\cdot(\varepsilon\frac{\partial}{\partial x}-2i\pi\theta^n)\bar{\phi}^{\varepsilon}dxdtd\mathbb{P}\\
&+\int_{\Omega}\int_0^T\int_D(c^{\varepsilon}-\lambda_n(\theta^n)+\varepsilon^2d^{\varepsilon})v_{\varepsilon}\bar{\phi}^{\varepsilon}dxdtd\mathbb{P}\\
&+\int_{\Omega}\int_0^T\int_D\varepsilon^2\tilde{g}^{\varepsilon}\bar{\phi}^{\varepsilon}dxdW(t)d\mathbb{P}=0.\\
\end{split}
\end{equation}
Passing to the two-scale limit term by term and applying (\ref{Pro.2}) and (\ref{Pro.1}), we obtain
\begin{equation}
-(\frac{\partial}{\partial y}+2i\pi\theta^n)\left(\sigma(y)(\frac{\partial}{\partial y}+2i\pi\theta^n)v^*\right)+c(y)v^*=\lambda_n(\theta^n)v^* \quad \text{in}\ \mathbb{T},
\end{equation}
for a.e. $\omega\in\Omega$. By the simplicity of $\lambda_n(\theta^n)$ of Hypothesis H.3., we know that there exists a scalar function $v(\omega, t, x)\in L^2(\Omega\times[0, T]\times D)$ such that
\begin{equation}
v^*(\omega, t, x, y)=v(\omega, t, x)\psi_n(y, \theta^n).\\
\end{equation}
Second step.\ We multiply (\ref{Equ.2}) by the complex conjugate of
 \begin{equation}\label{PsiConj.}
\Psi_{\varepsilon}=e^{i\frac{\lambda_n(\theta^n)t}{\varepsilon^2}}e^{2i\pi\frac{\theta^n\cdot x}{\varepsilon}}\left(\psi_n(\frac{x}{\varepsilon}, \theta^n)\phi(\omega, t, x)+\varepsilon\frac{\partial \phi}{\partial x}(\omega, t, x)\zeta(\frac{x}{\varepsilon})\right),
\end{equation}
where $\phi(\omega, t, x)$ is a smooth test function with compact support in  $[0, T]\times D$, and $\zeta(y)$ is the solution of (\ref{SolutionZeta}). Then we obtain
\begin{equation}\label{Variation1}
\begin{split}
&i\int_{\Omega}\int_Du^0_{\varepsilon}\bar{\Psi}_{\varepsilon}(t=0)dxd\mathbb{P}-i\int_{\Omega}\int_0^T\int_Du_{\varepsilon}\frac{\partial\bar{\Psi}}{\partial t}dxdtd\mathbb{P}\\
&+\int_{\Omega}\int_0^T\int_D\sigma^{\varepsilon}\frac{\partial u_{\varepsilon}}{\partial x} \cdot\frac{\partial\bar{\Psi}_{\varepsilon}}{\partial x}dxdtd\mathbb{P}\\
&+\frac{1}{\varepsilon^2}\int_{\Omega}\int_0^T\int_Dc^{\varepsilon}u_{\varepsilon}\bar{\Psi}_{\varepsilon}dxdtd\mathbb{P}\\
&+\int_{\Omega}\int_0^T\int_Dd^{\varepsilon}u_{\varepsilon}\bar{\Psi}_{\varepsilon}dxdtd\mathbb{P}\\
&+\int_{\Omega}\int_0^T\int_D\tilde{g}^{\varepsilon}\bar{\Psi}_{\varepsilon}dxdW(t)d\mathbb{P} =0.\\
\end{split}
\end{equation}
According to (\ref{PsiConj.}), we get
\begin{equation}\label{Variation2}
\begin{split}
&i\int_{\Omega}\int_Du^0_{\varepsilon}\bar{\Psi}_{\varepsilon}(t=0)dxd\mathbb{P}-i\int_{\Omega}\int_0^T\int_Dv_{\varepsilon}(\bar{\psi}_n^{\varepsilon}\frac{\partial\bar{\phi}}{\partial t}+\varepsilon\frac{\partial^2\bar{\phi}}{\partial x\partial t}\bar{\zeta}^{\varepsilon})dxdtd\mathbb{P}\\
&+\int_{\Omega}\int_0^T\int_D\sigma^{\varepsilon}\frac{\partial u_{\varepsilon}}{\partial x} \cdot\frac{\partial\bar{\Psi}_{\varepsilon}}{\partial x}dxdtd\mathbb{P}\\
&+\frac{1}{\varepsilon^2}\int_{\Omega}\int_0^T\int_D(c^{\varepsilon}-\lambda_n(\theta^n))v_{\varepsilon}\bar{\psi}_n^{\varepsilon}\bar{\phi}dxdtd\mathbb{P}\\
&+\frac{1}{\varepsilon}\int_{\Omega}\int_0^T\int_D(c^{\varepsilon}-\lambda_n(\theta^n))v_{\varepsilon}\frac{\partial\bar{\phi}}{\partial x}\bar{\zeta}^{\varepsilon}dxdtd\mathbb{P}\\
&+\int_{\Omega}\int_0^T\int_Dd^{\varepsilon}v_{\varepsilon}(\bar{\psi}_n^{\varepsilon}\bar{\phi}+\varepsilon\frac{\partial\bar{\phi}}{\partial x}\bar{\zeta}^{\varepsilon})dxdtd\mathbb{P}\\
&+\int_{\Omega}\int_0^T\int_D\tilde{g}^{\varepsilon}\bar{\Psi}_{\varepsilon}dxdW(t)d\mathbb{P} =0.\\
\end{split}
\end{equation}
After some algebra we find that
\begin{equation}\label{Div}
\begin{split}
\int_D\sigma^{\varepsilon}\frac{\partial u_{\varepsilon}}{\partial x} \cdot\frac{\partial\bar{\Psi}_{\varepsilon}}{\partial x}dx&=\int_D\sigma^{\varepsilon}(\frac{\partial}{\partial x}+2i\pi\frac{\theta^n}{\varepsilon})(\bar{\phi} v_{\varepsilon})\cdot(\frac{\partial}{\partial x}-2i\pi\frac{\theta^n}{\varepsilon})\bar{\psi}^{\varepsilon}_ndx\\
&+\varepsilon\int_D\sigma^{\varepsilon} (\frac{\partial}{\partial x}+2i\pi\frac{\theta^n}{\varepsilon})(\frac{\partial\bar{\phi}}{\partial x} v_{\varepsilon})\cdot(\frac{\partial}{\partial x}-2i\pi\frac{\theta^n}{\varepsilon})\bar{\zeta}^{\varepsilon}dx\\
&-\int_D\sigma^{\varepsilon}\frac{\partial\bar{\phi}}{\partial x} v_{\varepsilon}\cdot(\frac{\partial}{\partial x}-2i\pi\frac{\theta^n}{\varepsilon})\bar{\psi}^{\varepsilon}_ndx\\
&+\int_D\sigma^{\varepsilon}(\frac{\partial}{\partial x}+2i\pi\frac{\theta^n}{\varepsilon})(\frac{\partial\bar{\phi}}{\partial x} v_{\varepsilon})\cdot \bar{\psi}^{\varepsilon}_ndx\\
&-\int_D\sigma^{\varepsilon}v_{\varepsilon}\frac{\partial^2\bar{\phi}}{\partial x^2}\cdot \bar{\psi}^{\varepsilon}_ndx\\
&-\int_D\sigma^{\varepsilon}v_{\varepsilon}\frac{\partial^2\bar{\phi}}{\partial x^2}\cdot(\varepsilon\frac{\partial}{\partial x}-2i\pi\theta^n)\bar{\zeta}^{\varepsilon}dx\\
&+\int_D\sigma^{\varepsilon}\bar{\zeta}^{\varepsilon}(\varepsilon\frac{\partial}{\partial x}+2i\pi\theta^n)v_{\varepsilon}\cdot\frac{\partial^2\bar{\phi}}{\partial x^2}dx.\\
\end{split}
\end{equation}
Now, for any smooth compactly supported test function $\Phi$, we deduce from the definition of $\psi_n$ that
\begin{equation}\label{PsiDef.}
\int_D\sigma^{\varepsilon}(\frac{\partial}{\partial x}+2i\pi\frac{\theta^n}{\varepsilon})\psi_n^{\varepsilon}\cdot(\frac{\partial}{\partial x}-2i\pi\frac{\theta^n}{\varepsilon})\bar{\Phi}dx+\frac{1}{\varepsilon^2}\int_D(c^{\varepsilon}-\lambda_n(\theta^n))\psi_n^{\varepsilon}\bar{\Phi}dx=0,
\end{equation}
and from the definition of $\zeta$,
\begin{equation}\label{ZetaDef.}
\begin{split}
&\int_D\sigma^{\varepsilon}(\frac{\partial}{\partial x}+2i\pi\frac{\theta^n}{\varepsilon})\zeta^{\varepsilon}\cdot(\frac{\partial}{\partial x}-2i\pi\frac{\theta^n}{\varepsilon})\bar{\Phi}dx+\frac{1}{\varepsilon^2}\int_D(c^{\varepsilon}-\lambda_n(\theta^n))\zeta^{\varepsilon}\bar{\Phi}dx=\\
&\varepsilon^{-1}\int_D\sigma^{\varepsilon}(\frac{\partial}{\partial x}+2i\pi\frac{\theta^n}{\varepsilon})\psi_n^{\varepsilon}\cdot \bar{\Phi}dx-\varepsilon^{-1}\int_D\sigma^{\varepsilon}\psi_n^{\varepsilon}\cdot(\frac{\partial}{\partial x}-2i\pi\frac{\theta^n}{\varepsilon})\bar{\Phi}dx.\\
\end{split}
\end{equation}

Combining (\ref{Div}) with the other terms of the variational formulation of (\ref{Variation2}), we see that the first line of its right-hand side cancels out because of (\ref{PsiDef.}) with $\Phi=\bar{\phi}v_{\varepsilon}$, and the next three lines cancel out because of (\ref{ZetaDef.}) with $\Phi=\frac{\partial\bar{\phi}}{\partial x}v_{\varepsilon}$. On the other hand, we can pass to the limit in three last terms of (\ref{Div}). 

Finally, (\ref{Equ.2}) multiplied by $\bar{\Psi}_{\varepsilon}$ yields after simplification
\begin{equation}\label{Variation3}
\begin{split}
&i\int_{\Omega}\int_Du^0_{\varepsilon}\bar{\Psi}_{\varepsilon}(t=0)dxd\mathbb{P}-i\int_{\Omega}\int_0^T\int_Dv_{\varepsilon}(\bar{\psi}_n^{\varepsilon}\frac{\partial\bar{\phi}}{\partial t}+\varepsilon\frac{\partial^2\bar{\phi}}{\partial x\partial t}\bar{\zeta}^{\varepsilon})dxdtd\mathbb{P}\\
&-\int_{\Omega}\int_0^T\int_D\sigma^{\varepsilon}v_{\varepsilon}\frac{\partial^2\bar{\phi}}{\partial x^2}\cdot \bar{\psi}^{\varepsilon}_ndxdtd\mathbb{P}\\
&-\int_{\Omega}\int_0^T\int_D\sigma^{\varepsilon}v_{\varepsilon}\frac{\partial^2\bar{\phi}}{\partial x^2}\cdot(\varepsilon\frac{\partial}{\partial x}-2i\pi\theta^n)\bar{\zeta}^{\varepsilon}dxdtd\mathbb{P}\\
&+\int_{\Omega}\int_0^T\int_D\sigma^{\varepsilon}\bar{\zeta}^{\varepsilon}(\varepsilon\frac{\partial}{\partial x}+2i\pi\theta^n)v_{\varepsilon}\cdot\frac{\partial^2\bar{\phi}}{\partial x^2}dxdtd\mathbb{P}\\
&+\int_{\Omega}\int_0^T\int_Dd^{\varepsilon}v_{\varepsilon}(\bar{\psi}_n^{\varepsilon}\bar{\phi}+\varepsilon\frac{\partial\bar{\phi}}{\partial x}\bar{\zeta}^{\varepsilon})dxdtd\mathbb{P}\\
&+\int_{\Omega}\int_0^T\int_D\tilde{g}^{\varepsilon}\bar{\Psi}_{\varepsilon}dxdW(t)d\mathbb{P} =0.\\
\end{split}
\end{equation}
Passing to the two-scale limit in each term of (\ref{Variation3}) gives
\begin{equation}\label{LimitEquation}
\begin{split}
&i\int_{\Omega}\int_D\int_{\mathbb{T}}\psi_nv^0\bar{\psi}_n\bar{\phi}(t=0)dydxd\mathbb{P}-i\int_{\Omega}\int_0^T\int_D\int_{\mathbb{T}}\psi_nv\bar{\psi}_n\frac{\partial\bar{\phi}}{\partial t}dydxdtd\mathbb{P}\\
&-\int_{\Omega}\int_0^T\int_D\int_{\mathbb{T}}\sigma\psi_nv\frac{\partial^2\bar{\phi}}{\partial x^2}\cdot \bar{\psi}_ndydxdtd\mathbb{P}\\
&-\int_{\Omega}\int_0^T\int_D\int_{\mathbb{T}}\sigma\psi_nv\frac{\partial^2\bar{\phi}}{\partial x^2}\cdot(\frac{\partial}{\partial x}-2i\pi\theta^n)\bar{\zeta}dydxdtd\mathbb{P}\\
&+\int_{\Omega}\int_0^T\int_D\int_{\mathbb{T}}\sigma\bar{\zeta}(\frac{\partial}{\partial x}+2i\pi\theta^n)\psi_nv\cdot\frac{\partial^2\bar{\phi}}{\partial x^2}dydxdtd\mathbb{P}\\
&+\int_{\Omega}\int_0^T\int_D\int_{\mathbb{T}}d(x, y)\psi_nv\bar{\psi}_n\bar{\phi}dydxdtd\mathbb{P}\\
&+\int_{\Omega}\int_0^T\int_D\int_{\mathbb{T}}\tilde{g}(y)\psi_n\bar{\psi}_n\bar{\phi}dydxdW(t)d\mathbb{P} =0.\\
\end{split}
\end{equation}
Recalling the normalization $\int_{\mathbb{T}}|\psi_n|^2dy=1$, and introducing
\begin{equation}\label{Sigma}
\sigma_n^*=\int_{\mathbb{T}}(\sigma\psi_n\bar{\psi}_n+\sigma\psi_n(\frac{\partial}{\partial y}-2i\pi\theta^n)\bar{\zeta}-\sigma\bar{\zeta}(\frac{\partial}{\partial y}+2i\pi\theta^n)\psi_n)dy,
\end{equation}
and $d^*_n(x)=\int_{\mathbb{T}}d(x, y)|\psi_n(y)|^2dy$, $g_n^*=\int_{\mathbb{T}}\tilde{g}(y)|\psi_n(y)|^2dy$, equation (\ref{LimitEquation}) is equivalent to
\begin{equation}
\begin{split}
&i\int_{\Omega}\int_Dv^0\bar{\phi}dxd\mathbb{P}-i\int_{\Omega}\int_0^T\int_Dv\frac{\partial\bar{\phi}}{\partial t}dxdtd\mathbb{P}-\int_{\Omega}\int_0^T\int_D\sigma_n^*v\cdot\frac{\partial^2\bar{\phi}}{\partial x^2}dxdtd\mathbb{P}\\
&+\int_{\Omega}\int_0^T\int_Dd^*(x)v\bar{\phi}dxdtd\mathbb{P}+\int_{\Omega}\int_0^T\int_Dg_n^*\bar{\phi}dxdW(t)d\mathbb{P}=0,\\
\end{split}
\end{equation}
which is a very weak form of the homogenized equation (\ref{Homo.1}). The compatibility condition of (\ref{SolutionChi}) for the second derivative of $\psi_n$ shows that the $\sigma_n^*$, defined by (\ref{Sigma}), is indeed equal to $\frac{1}{8\pi^2}\frac{\partial^2\lambda_n(\theta^n)}{\partial \theta^2}$, 

This completes the proof.
\end{proof}

\begin{remark}
In the context of quantum mechanics or solid state physics Theorem 1 is called an effective mass theorem \cite{My90}. More precisely, the $(\sigma_n^*)^{-1}$ is the effective mass of an electron in the $n^{th}$ band of a periodic crystal (characterized by the periodic metric $\sigma(y)$ and the periodic potential $c(y)$).
\end{remark}

\section{Homogenization with Multiplicative Noise}
We now study the homogenization of the following Schr\"odinger equation with multiplicative white noise
\begin{equation}\label{Equ.3}
\begin{cases}
$$i\frac{\partial u_{\varepsilon}}{\partial t} - \frac{\partial}{\partial x}(\sigma(\frac{x}{\varepsilon})\frac{\partial u_{\varepsilon}}{\partial x} )+(\varepsilon^{-2}c(\frac{x}{\varepsilon})+d(x,\frac{x}{\varepsilon}))u_{\varepsilon}+g(\frac{x}{\varepsilon})u_{\varepsilon}\frac{dW(t)}{dt}=0  \quad \text{in} \ D\times [0,T], \\
u_\varepsilon=0 \quad \text{on} \ \partial D\times [0,T], \\
u_\varepsilon(0,x)=u^0_{\varepsilon}(x)  \quad \text{in} \ D.$$
\end{cases}
\end{equation}

{\bf Hypothesis  H.5. } We assume that the function $g$ is real measurable bounded periodic function.

Transforming into a pathwise problem, exploiting the absence of noise and using Galerkin approximations and compact embedding results, we can obtain a priori estimates, existence and uniqueness of the variational solution of stochastic linear Schr\"odinger equation with multiplicative white noise. For the following lemma, see \cite{Ke15}.
\begin{lemma}
Assume (H.1., H.2., H.5.). For every $\varepsilon>0$, $u^0_{\varepsilon}\in H^1(D)$, and $T>0$, there exists a unique variational solution $u_{\varepsilon}\in L^2(\Omega;C([0,T]);L^2(D))\bigcap L^2(\Omega\times[0,T];H^1(D))$ of stochastic Schr\"odinger equation (\ref{Equ.3}) in the following sense:
\begin{equation}
\begin{split}
(u_{\varepsilon},v)&=(u^0_{\varepsilon},v)-i\int_0^t(\sigma(\frac{x}{\varepsilon})\frac{\partial u_{\varepsilon}}{\partial x}, \frac{\partial v}{\partial x})ds+i\int_0^t\left((\varepsilon^{-2}c(\frac{x}{\varepsilon})+d(x,\frac{x}{\varepsilon}))u_{\varepsilon},v\right)ds\\
&+i\int_0^t(g(\frac{x}{\varepsilon})u_{\varepsilon}, v)dW(s),
\end{split}
\end{equation}
for a.e. $\omega\in\Omega$, all $t\in[0,T]$ and for all $v\in H^1(D)$. Moreover, there exists a constant $C_T$ that depends on $T$ such that
\begin{equation}
\mathbb{E}(\sup_{0\leqslant t\leqslant T}\|u_{\varepsilon}\|^2+\varepsilon^2\int_0^T\|u_{\varepsilon}\|^2_{H^1}dt)<C_T.
\end{equation}
\end{lemma}

\begin{theorem}
Assume (H.1.-H.3., H.5.) and that the initial data $u_{\varepsilon}^0\in H^1(D)$ is of the form
\begin{equation}
u_{\varepsilon}^0(x)=\psi_n(\frac{x}{\varepsilon},\theta^n)e^{2i\pi\frac{\theta^n\cdot x}{\varepsilon}}v^0(x),
\end{equation}
with $v^0\in H^1(\mathbb{R})$. The solution of (1) can be written as
\begin{equation}
u_{\varepsilon}(t,x)=e^{i\frac{\lambda_n(\theta^n)t}{\varepsilon^2}}e^{2i\pi\frac{\theta^n\cdot x}{\varepsilon}}v_{\varepsilon}(t, x),
\end{equation}
where $v_{\varepsilon}$ two-scale converges  to $\psi_n(y,\theta^n)v(t, x)$, uniformly on compact time intervals in $\mathbb{R}^+$, and $v$ is the unique
solution of the homogenized Schr\"odinger equation
\begin{equation}\label{Homo.2}
\begin{cases}
$$i\frac{\partial v}{\partial t} - \frac{\partial}{\partial x}(\sigma_n^*\frac{\partial v}{\partial x})+d_n^*(x)v+g_n^*v\frac{dW(t)}{dt}=0  \quad \text{in} \ D\times [0,T], \\
v=0 \quad \text{on} \ \partial D\times [0,T], \\
v(0,x)=v^0(x)  \quad \text{in} \ D.$$
\end{cases}
\end{equation}
with $\sigma_n^*=\frac{1}{8\pi^2}\frac{\partial^2\lambda_n(\theta^n)}{\partial\theta^2}$, $d^*_n(x)=\int_{\mathbb{T}}d(x, y)|\psi_n(y)|^2dy$, and $g_n^*=\int_{\mathbb{T}}g(y)|\psi_n(y)|^2dy$.
\end{theorem}

\begin{proof}
The proof is similar to that of Theorem 1. First of all, by Lemma 2, the same convergence of (\ref{Pro.2}) and (\ref{Pro.1}) hold. As in the first step of the proof of Theorem 1, we obtain
\begin{equation}
\begin{split}
i\varepsilon^2&\int_{\Omega}\int_Du^0_{\varepsilon}\bar{\phi}^{\varepsilon}e^{-2i\pi\frac{\theta^n\cdot x}{\varepsilon}}dxd\mathbb{P}-i\varepsilon^2\int_{\Omega}\int_0^T\int_Dv_{\varepsilon}\frac{\partial\bar{\phi}^{\varepsilon}}{\partial t}dxdtd\mathbb{P}\\
&+\int_{\Omega}\int_0^T\int_D\sigma^{\varepsilon}(\varepsilon\frac{\partial}{\partial x}+2i\pi\theta^n)v_{\varepsilon}\cdot(\varepsilon\frac{\partial}{\partial x}-2i\pi\theta^n)\bar{\phi}^{\varepsilon}dxdtd\mathbb{P}\\
&+\int_{\Omega}\int_0^T\int_D(c^{\varepsilon}-\lambda_n(\theta^n)+\varepsilon^2d^{\varepsilon})v_{\varepsilon}\bar{\phi}^{\varepsilon}dxdtd\mathbb{P}=0.\\
&+\int_{\Omega}\int_0^T\int_D\varepsilon^2g^{\varepsilon}v_{\varepsilon}\bar{\phi}^{\varepsilon}dxdW(t)d\mathbb{P}=0.\\
\end{split}
\end{equation}

Similarly, we obtain that the sequence
\begin{equation}
v_{\varepsilon}(\omega,t,x)=u_{\varepsilon}(\omega, t, x)e^{-i\frac{\lambda_n(\theta^n)t}{\varepsilon^2}}e^{-2i\pi\frac{\theta^n\cdot x}{\varepsilon}}\xrightarrow{2-s}v(\omega, t, x)\psi_n(y, \theta^n) \quad \text{in} \  \Omega\times[0,T]\times D,\\
\end{equation}
then, in a second step we multiply (\ref{Equ.3}) by the complex conjugate of (\ref{PsiConj.}). After integration by parts and some algebra similar to that in the proof of Theorem 1, we obtain
\begin{equation}\label{Variation4}
\begin{split}
&i\int_{\Omega}\int_Du^0_{\varepsilon}\bar{\Psi}_{\varepsilon}(t=0)dxd\mathbb{P}-i\int_{\Omega}\int_0^T\int_Dv_{\varepsilon}(\bar{\psi}_n^{\varepsilon}\frac{\partial\bar{\phi}}{\partial t}+\varepsilon\frac{\partial^2\bar{\phi}}{\partial x\partial t}\bar{\zeta}^{\varepsilon})dxdtd\mathbb{P}\\
&-\int_{\Omega}\int_0^T\int_D\sigma^{\varepsilon}v_{\varepsilon}\frac{\partial^2\bar{\phi}}{\partial x^2}\cdot \bar{\psi}^{\varepsilon}_ndxdtd\mathbb{P}\\
&-\int_{\Omega}\int_0^T\int_D\sigma^{\varepsilon}v_{\varepsilon}\frac{\partial^2\bar{\phi}}{\partial x^2}\cdot(\varepsilon\frac{\partial}{\partial x}-2i\pi\theta^n)\bar{\zeta}^{\varepsilon}dxdtd\mathbb{P}\\
&+\int_{\Omega}\int_0^T\int_D\sigma^{\varepsilon}\bar{\zeta}^{\varepsilon}(\varepsilon\frac{\partial}{\partial x}+2i\pi\theta^n)v_{\varepsilon}\cdot\frac{\partial^2\bar{\phi}}{\partial x^2}dxdtd\mathbb{P}\\
&+\int_{\Omega}\int_0^T\int_Dd^{\varepsilon}v_{\varepsilon}(\bar{\psi}_n^{\varepsilon}\bar{\phi}+\varepsilon\frac{\partial\bar{\phi}}{\partial x}\bar{\zeta}^{\varepsilon})dxdtd\mathbb{P}\\
&+\int_{\Omega}\int_0^T\int_Dg^{\varepsilon}v_{\varepsilon}(\bar{\psi}_n^{\varepsilon}\bar{\phi}+\varepsilon\frac{\partial\bar{\phi}}{\partial x}\bar{\zeta}^{\varepsilon})dxdW(t)d\mathbb{P} =0.\\
\end{split}
\end{equation}
Passing to the two-scale limit in each term of (\ref{Variation4}) gives a variational formulation of (\ref{Homo.2}).
\end{proof}
This completes the proof.

\section{Outlook}
There are many possibilities to deepen these deliberations. One research perspective is to regard the problem over an unbounded domain or even a multi-dimensional one. Furthermore, an encouraging outlook is the discussion of more general nonlinear power-terms.

\end{document}